\documentclass{amsart}
\usepackage{amssymb}  
\usepackage{amsmath} 
\usepackage{amsthm}  
\usepackage{bm} 
\usepackage{stmaryrd} 

\usepackage[pic]{xy}   

\advance\hoffset by -1in
\advance\voffset by -1in

\newtheorem{thm}{Theorem}
\newtheorem{lem}[thm]{Lemma}

\newtheorem{cor}[thm]{Corollary}
\newtheorem{prob}{Problem}
\newtheorem{conj}{Conjecture}

\newcommand{\vf}{\varphi}

\renewcommand{\r}{\rho}
\renewcommand{\le}{\leqslant}
\renewcommand{\ge}{\geqslant}

\renewcommand{\a}{\alpha}

\renewcommand{\t}{\tau}

\newcommand{\g}{\gamma}

\newcommand{\ov}{\overline}
\newcommand{\FF} {\mathbb{F}}

\newcommand{\ZZ} {\mathbb{Z}}

\newcommand{\Ker}{\mathop\mathrm{Ker}\nolimits}

\newcommand{\Ext}{\mathop\mathrm{Ext}\nolimits}

\newcommand{\Hom}{\mathop\mathrm{Hom}\nolimits}

\newcommand{\PSL}{\mathop\mathrm{PSL}\nolimits}

\newcommand{\Syl}{\mathop\mathrm{Syl}\nolimits}
\newcommand{\SL}{\mathop\mathrm{SL}\nolimits}

\renewcommand{\Im}{\mathop\mathrm{Im}\nolimits}

\renewcommand{\P}{\mathcal{P}}

\renewcommand{\L}{\mathrm{L}}

\title{Subextensions for a permutation $\PSL_2(q)$-module}
\author{Andrei V. Zavarnitsine}
\address{\textup{\scriptsize
Andrei V. Zavarnitsine\\
Group Theory Lab.\\
Sobolev Institute of Mathematics\\
4, Koptyug av.\\
630090, Novosibirsk, Russia\\
}}
\email{zav@math.nsc.ru}

\thanks{Supported by the Russian Foundation for Basic Research
(projects 11--01--00456, 11--01--91158, 12--01--90006, 13--01--00505);
by the Federal Target Grant ``Scientific and
educational personnel of innovation Russia'' (contract 14.740.11.0346).}
\date{}

\begin{document}
\begin{abstract} Using cohomological methods, we prove the existence of a subgroup isomorphic to $\SL_2(q)$, $q\equiv -1\pmod 4$, in the permutation module for $\PSL_2(q)$ in characteristic $2$ that arises from the action on the projective line. A similar problem for $q\equiv 1\pmod 4$ is reduced to the determination of certain first cohomology groups.


\end{abstract}
\maketitle

\section{Introduction}

We denote by $\FF_q$ a finite field of order $q$ and by $\ZZ_n$ a cyclic group of order $n$.

Let $q$ be an odd prime power and let $G=\PSL_2(q)$. From the Universal embedding \cite[Theorem 2.6.A]{dm}, it follows that the regular wreath product $\ZZ_2 \wr G$ contains a subgroup isomorphic to $\SL_2(q)$.
It is of interest to know if the same is true for a permutation wreath product that is not necessarily regular. In particular, let  $\r$ be the natural permutation representation of $G$ of degree $q+1$ on the projective line over $\FF_q$. The following problem arose in the research~\cite{mr}.

\begin{prob} \label{mpr}
Does the permutation wreath product $\ZZ_2 \wr_\r G$ contain a subgroup isomorphic to $\SL_2(q)$?
\end{prob}
Although stated in purely group-theoretic terms, this problem is cohomologic in nature. In the next section, we reformulate a generalized version of this question as an assertion about a homomorphism between second cohomology groups of group modules. We then apply some basic theory to obtain the following

\begin{thm} \label{main}
If $q\equiv -1 \pmod 4$ then the answer to Problem \ref{mpr} is affirmative.
\end{thm}

It seems that the case $q\equiv 1 \pmod 4$ is more complicated and requires some deeper considerations than those presented here. We put forward
\begin{conj} \label{cn}
If $q\equiv 1 \pmod 4$ then $\SL_2(q)$ is not embedded in $\ZZ_2 \wr_\r G$.
\end{conj}
For small values $q=5,9,13,17$, Conjecture \ref{cn} was confirmed using a computer. In the last section, we show how to reduce this problem to the determination of the first cohomology groups $H^1(G,U_{\pm})$, where $U_+$ and $U_-$ are the two nontrivial absolutely irreducible $G$-modules in the principal $2$-block.

\section{Subextensions for group modules}

Let $G$ be a group and let $L,M$ be right $G$-modules. Let
\begin{equation}\label{emb}
0\to L \to M
\end{equation}
and
$$
1\to M \to E \stackrel{\pi}\to G \to 1
$$
be exact sequences of modules and groups, where the conjugation action of $E$ on $M$ agrees with the $G$-module structure, i.\,e. $m^e=m\cdot\pi(e)$ for all $m\in M$ and $e\in E$, and we identify $M$ with its image in $E$.
Then we call $E$ an {\em extension of $M$ by $G$}. It is natural to ask if there is a subgroup $H\le E$ such that
\begin{equation}\label{hpro}
H\cap M = L, \qquad HM = E,
\end{equation}
where we implicitly identify $L$ with its image in $M$. A subgroup $H$ with these properties is itself an extension of $L$ by $G$, and will thus be called a {\em subextension} of $E$ that corresponds to the embedding (\ref{emb}). The classification of all such subextensions of $E$ (whenever they exist) up to equivalence is also of interest.

Recall that extensions $H_1,H_2$ of $L$ by $G$ are {\em equivalent}
if there is a homomorphism $\a$ that makes the diagram

$$
\xymatrix@=1em{
 & H_1  \ar@{.>}[dd]^{ \textstyle{\a} }   \ar[dr] & \\
L\ar[ur]\ar[dr]& & G \\
 & H_2 \ar[ur] &
}
$$
commutative. It is known \cite{gr} that the equivalence classes of such extensions are in a one-to-one correspondence with (thus are {\em defined by}) the elements of the second cohomology group $H^2(G,L)$. Furthermore, the sequence (\ref{emb}) gives rise to a homomorphism
\begin{equation}\label{hhom}
H^2(G,L)\stackrel{\varphi}{\to} H^2(G,M).
\end{equation}
The following assertion is nothing more than an interpretation of this homomorphism in group-theoretic terms.

\begin{lem} \label{bas} Let $L,M$ be $G$-modules and $E$ an extension as specified above. Let $\bar\g\in H^2(G,M)$ be the element that defines $E$. Then
the set of elements of $H^2(G,L)$ that define the subextensions $H$ of $E$ corresponding to the embedding {\rm (\ref{emb})} coincides with $\vf^{-1}(\bar\g)$, where $\vf$ is the induced homomorphism {\rm (\ref{hhom})}. In particular,
$E$ has such a subextension $H$ if and only if $\bar\g\in \Im\vf$.
\end{lem}
\begin{proof} Let $H$ be a required subextension of $E$. Choose a transversal $\t: G \to H$ of $L$ in $H$. Then,
for all $g_1,g_2\in G$, we have
$\t(g_1)\t(g_2)=\t(g_1g_2)\beta(g_1,g_2)$ for a $2$-cocycle $\beta\in Z^2(G,L)$ and the element $\bar\beta=\beta+B^2(G,L)$ of $H^2(G,L)$ defines $H$. Let $\g$ be the composition of $\beta$ with the embedding (\ref{emb}). Then $\g\in Z^2(G,M)$ arises from the same transversal $\t$ (composed with the embedding $H\to E$), hence $\g+B^2(G,M)$ is the element of $H^2(G,M)$ that defines $E$ which is $\bar\gamma$. Therefore, $\vf(\bar\beta)=\bar\gamma$.

Conversely, let $\vf(\bar\beta)=\bar\gamma$ for some $\bar\beta\in H^2(G,L)$. Then there is a representative $2$-cocycle $\gamma \in Z^2(G,M)$ whose values lie in $L$ and which, when viewed as a map $G\times G\to L$, is a $2$-cocycle $\beta\in Z^2(G,L)$ representative for $\bar\beta$. Now $E$ can be identified with the set of pairs $(g,m)$ with $g\in G$, $m\in M$
subject to the multiplication
$$
(g_1,m_1)(g_2,m_2)=(g_1g_2,m_1\cdot g_2+m_2+\beta(g_1,g_2))
$$
and, if we set $H=\{(g,m)\mid g\in G, m\in L \}$, then $H$ is clearly a subextension of $E$ defined by $\bar\beta$.
\end{proof}

It is known that the zero element of $H^2(G,M)$ defines the split extension (which fact is also a particular case of Lemma \ref{bas} with $L=0$). Therefore, we have

\begin{cor} \label{cors} Let $L,M$ be $G$-modules as above and let  $E$ be the split extension of $M$ by $G$. Then the subextensions of $E$ that correspond to the embedding (\ref{emb}) are defined by the elements of $\,\Ker\vf$, where $\vf$ is the induced homomorphism {\rm (\ref{hhom})}.
\end{cor}

\section{Notation and auxiliary results}

Basic facts of homological algebra can be found in \cite{gr,we}. For abelian groups $A$ and $B$, we denote $\Hom(A,B) = \Hom_\ZZ(A,B)$ and $\Ext(A,B) = \Ext_\ZZ^1(A,B)$.

\begin{lem}[The Universal Coefficient Theorem for Cohomology, {\cite[Theorem 3]{gr}}] \label{ucoe} For all $i\ge 1$ and every trivial $G$-module $A$,
$$
H^i(G,A)\cong\Hom(H_i(G,\ZZ),A)\oplus\Ext(H_{i-1}(G,\ZZ),A).
$$
\end{lem}

\begin{lem}[Shapiro's lemma, {\cite[\S 6.3]{we}}] \label{shap} Let $H\le G$ with $|G:H|$ finite. If $V$ is an $H$-module and $i\ge 0$ then $H^i(G,V^G)\cong H^i(H,V)$, where $V^G$ is the induced $G$-module.
\end{lem}

Given a group $G$, we denote by $M(G)$ the Schur multiplier of $G$.
If $A$ is a finite abelian group and $p$ a prime then $A_{(p)}$ denotes the $p$-primary component of $A$.

\begin{lem}\cite[Theorem 25.1]{hu} \label{tim} Let $G$ be a finite group, $p$ a prime, and let $P\in\Syl_p(G)$.
Then $M(G)_{(p)}$ is isomorphic to a subgroup of $M(P)$.
\end{lem}

\section{Projective action of $\PSL_2(q)$}

We denote $G=\PSL_2(q)$ for $q$ odd. Let $\P$ be the projective line over $\FF_q$ and let $V$ be the permutation $\FF_2G$-module that corresponds to the natural action of $G$ on $\P$. The sum of the basis vectors of $V$, which are permuted by $G$, spans a 1-dimensional submodule $I$, and we have the exact sequence
\begin{equation}\label{esp}
0\to I \to V \to W \to 0,
\end{equation}
where $W\cong V/I$. The following result clarifies the composition structure of the module $V$.
Let $k=\ov{\FF}_2$ be the algebraic closure of $\FF_2$.
\begin{lem}[{\cite[Lemma 1.6]{bu}}] \label{bur} In the notation above, $I$ is the unique minimal submodule of $V$ and $W$ has a unique maximal submodule $U$  such that
\begin{equation}\label{bst}
0\to U \to W \to I \to 0
\end{equation}
is a nonsplit short exact sequence. Moreover, $U\otimes k = U_{+}\oplus U_{-}$, where $U_{+}$ and $U_{-}$ are the two nontrivial absolutely irreducible  $kG$-modules in the principal $2$-block of $G$.
\end{lem}

Using the knowledge of the Schur multiplier of $G$ we can determine $H^2(G,I)$.
\begin{lem} \label{2hom}
Let $q$ be an odd prime power. For $\,\PSL_2(q)$ acting trivially on $\ZZ_2$, we have
$$
H^2(\PSL_2(q), \ZZ_2)\cong \ZZ_2.
$$
\end{lem}
\begin{proof} Applying Lemma \ref{ucoe} for the trivial action of $G=\PSL_2(q)$ on $\ZZ_2$, we have
\begin{equation}\label{uct}
H^2(G,\ZZ_2)\cong\Hom(H_2(G,\ZZ),\ZZ_2)\oplus\Ext(H_1(G,\ZZ),\ZZ_2).
\end{equation}
Since $H_2(G,\ZZ)=M(G)$, according to \cite[Theorem 25.7]{hu}, we have
$$
H_2(G,\ZZ)=\left\{
             \begin{array}{rl}
               \ZZ_2, & q\ne 9, \\
               \ZZ_6, & q=9.
             \end{array}
           \right.
$$
It follows that $\Hom(H_2(G,\ZZ),\ZZ_2)\cong\ZZ_2$. Since the first integral homology group $H_1(G,\ZZ)$ is isomorphic to the abelianization $G/G'$, we have
$$
H_1(G,\ZZ)=\left\{
             \begin{array}{rl}
               0, & q\ne 3, \\
               \ZZ_3, & q=3.
             \end{array}
           \right.
$$
Therefore, we always have $\Ext(H_1(G,\ZZ),\ZZ_2)=0$, and the claim follows.
\end{proof}

We now determine the group $H^2(G,V)$.
\begin{lem} \label{hgv} Let $V$ be the above-defined permutation module. Then we have
$$
H^2(G,V)\cong\left\{
             \begin{array}{rl}
               0, & q\equiv -1\pmod 4, \\
               \ZZ_2, & q\equiv 1\pmod 4.
             \end{array}
           \right.
$$
\end{lem}
\begin{proof} Since the action of $G$ on $\P$ is transitive, we have $V\cong T^G$, where $T$ is the principal  $\FF_2H$-module for a point stabilizer $H\le G$. By Lemma \ref{shap},  $H^2(G,V)\cong H^2(H,T)$. We have $H\cong \FF_q\leftthreetimes\ZZ_{(q-1)/2}$, a Frobenius group. If $q\equiv -1\pmod 4$, the order $|H|$ is odd. By the Schur-Zassenhaus theorem, every extension of a $2$-group by $H$ splits, which yields $H^2(H,T)=0$. Suppose that $q\equiv 1\pmod 4$. Let $P\in \Syl_2(H)$. Lemma \ref{tim} implies that $H_2(H,\ZZ)_{(2)}$ is a subgroup of $H_2(P,\ZZ)$ which is $0$, since cyclic groups have trivial Schur multiplier. Therefore,
$$
\Hom(H_2(H,\ZZ),\ZZ_2)=\Hom(H_2(H,\ZZ)_{(2)},\ZZ_2)=0.
$$
Note also that $H_1(H,\ZZ)=\ZZ_{(q-1)/2}$.
Now, $H$ acts trivially on $T\cong \ZZ_2$, so we can use again the universal coefficient formula (\ref{uct}) to obtain
$$
H^2(H,T)=\Ext(\ZZ_{(q-1)/2},\ZZ_2)\cong \ZZ_2,
$$
since $(q-1)/2$ is even by assumption. This completes the proof.
\end{proof}

\section{Proof of Theorem \ref{main}}
Given a permutation representation $\r$ of $G=\PSL_2(q)$ as in the introduction, observe that the permutation wreath product $E=\ZZ_2 \wr_\r G$ is isomorphic to
the split extension of the permutation $\FF_2G$-module $V$ by $G$. Since $\SL_2(q)$ is an extension of the principal $\FF_2G$ module $I$ by $G$, and $I$ is a unique minimal submodule of $V$ by Lemma \ref{bur}, it follows that
Problem \ref{mpr} is equivalent to the question of whether $\SL_2(q)$ is a subextension of $E$
that corresponds to the embedding $I\to V$. Corollary \ref{cors} implies that such subextensions are defined by the elements of $\Ker\vf$, where $\vf$ is the induced homomorphism
\begin{equation}\label{fiv}
H^2(G,I)\stackrel{\varphi}{\to} H^2(G,V).
\end{equation}
For $q\equiv -1\pmod 4$, we have $H^2(G,V)=0$ by Lemma \ref{hgv}. Thus, $\Ker\vf=H^2(G,I)\cong \ZZ_2$ by Lemma \ref{2hom}. The nonidentity element of $H^2(G,I)$ defines the unique nonsplit extension $\SL_2(q)$ of $I$ by $G$, which is therefore a subextension of $E$. The proof is complete.

\section{Reduction in the case $q\equiv 1\pmod 4$}
The above argument does not clarify what $\Ker\vf$ is if $q\equiv 1\pmod 4$, because in this case $H^2(G,V)\cong \ZZ_2$ by Lemma \ref{hgv}. We will consider the long sequence
\begin{equation}\label{liv}
H^1(G,I)\to H^1(G,V)  \to H^1(G,W) \stackrel{\delta}{\to} H^2(G,I)\stackrel{\varphi}{\to} H^2(G,V)
\end{equation}
induced by (\ref{esp}). Since this sequence is exact, it follows that $\Im\delta=\Ker\vf$, and we might as well study the connecting homomorphism $\delta$.
\begin{lem} \label{h1} In the above notation, we have
\begin{enumerate}
  \item[$(i)$] $H^1(G,I)=0$;
  \item[$(ii)$] $H^1(G,V)\cong \left\{
             \begin{array}{rl}
               0, & q\equiv -1\pmod 4, \\
               \ZZ_2, & q\equiv 1\pmod 4.
             \end{array}
           \right.$
\end{enumerate}
\end{lem}
\begin{proof} $(i)$ This follow from the interpretation of $|H^1(G,I)|$ as the number of conjugacy classes of complements to $I$ in $I\times G$, which is clearly 1, or from Lemma \ref{ucoe} for $i=1$.

$(ii)$ We again use the fact that $V\cong T^G$ as in the proof of Lemma \ref{hgv}, where $T$ is the principal $\FF_2H$-module for the Borel subgroup $H\le G$. We have $H^1(G,V)\cong H^1(H,T)$ by Lemma \ref{shap}. Assume that $q\equiv -1\pmod 4$. (Although we have covered this case in the previous sections, we still consider it for the sake of completeness.) We have that $|H|$ is odd and $|H^1(H,T)|=0$ by the Schur-Zassenhaus theorem. Let $q\equiv 1\pmod 4$. By Lemma \ref{ucoe},
$$
H^1(H,T)\cong\Hom(H_1(H,\ZZ),T)\oplus\Ext(H_0(H,\ZZ),T),
$$
where the second summand is zero as $H_0(H,\ZZ)\cong \ZZ$ and $\Ext(\ZZ,T)=0$, since $\ZZ$ is projective. Now, $H_1(H,\ZZ)\cong \ZZ_{(q-1)/2}$ and $\Hom(\ZZ_{(q-1)/2},\ZZ_2)=\ZZ_2$, since $(q-1)/2$ is even by assumption.
The claim follows.
\end{proof}
Lemma \ref{h1} implies that $\Im\delta\cong H^1(G,W)/H^1(G,V)$ and so it remains to determine $H^1(G,W)$.
To this end, we consider the long exact sequence
\begin{equation}\label{biv}
H^0(G,W)\to H^0(G,I) \to H^1(G,U) \to H^1(G,W) \to H^1(G,I)
\end{equation}
induced by (\ref{bst}). We have $H^1(G,I)=0$ by Lemma~\ref{h1} and $H^0(G,W)=0$ by Lemma~\ref{bur}. Therefore,
$$
H^1(G,W)\cong H^1(G,U)/H^0(G,I),
$$
and, since $H^0(G,I)=\ZZ_2$ is known, in view of Lemma \ref{bur} it remains to determine the groups $H^1(G,U_{\pm})$ for the absolutely irreducible modules $U_{\pm}$. This will settle the case $q\equiv 1\pmod 4$.

{\em Acknowledgement.} The author is thankful to Prof. D. O. Revin for drawing attention to Problem \ref{mpr} and discussing the content of this paper.

\end{document}